\documentclass[11pt]{amsart}

\usepackage{enumerate}
\usepackage{amsfonts}
\usepackage{amssymb}
\usepackage{amsmath}
\usepackage{tikz}
\usepackage{pgf}
\usepackage{cite}
\usetikzlibrary{arrows,topaths,intersections,calc}
\usetikzlibrary{decorations.markings}
\usepackage{stmaryrd}
\usepackage{filecontents}
\usepackage{cleveref}
\usepackage[mathscr]{euscript}

\newtheorem{theorem}{Theorem}[section]
\newtheorem*{thm}{Theorem}
\newtheorem{lemma}[theorem]{Lemma}
\newtheorem{proposition}[theorem]{Proposition}
\newtheorem{corollary}[theorem]{Corollary}

\theoremstyle{definition}
\newtheorem{definition}[theorem]{Definition}

\theoremstyle{remark}
\newtheorem{remark}[theorem]{Remark}

\numberwithin{equation}{section}

\DeclareMathOperator{\alg}{alg}

\DeclareMathOperator{\Gal}{Gal}

\DeclareMathOperator{\Log}{Log}

\newcommand{\mfrak}[1]{\mathfrak{#1}}
\newcommand{\mcal}[1]{\mathcal{#1}}
\newcommand{\mbb}[1]{\mathbb{#1}}

\newcommand{\Cl}{\mathrm{Cl}}

\newcommand{\ZP}{\mbb{Z}_p}

\begin{document}

\title[$C^1(\ZP)^*$ and $\ZP$-extensions]{On a construction of $C^1(\mbb{Z}_p)$ functionals from $\mbb{Z}_p$-extensions of algebraic number fields}


\author{Timothy All}
\email{timothy.all@rose-hulman.edu}
\address{5500 Wabash Ave, Terre Haute, IN 47803}


\author{Bradley Waller}
\email{waller@math.osu.edu}
\address{231 W 18th Ave, Columbus, OH 43210}


\subjclass[2010]{}{}


\date{}


\begin{abstract}
Let $k$ be any number field and $k_{\infty}/k$ any $\mbb{Z}_p$-extension. We construct a natural $\Lambda= \mbb{Z}_p\llbracket T-1 \rrbracket$-morphism from $\varprojlim k_n^{\times} \otimes_{\mbb{Z}} \mbb{Z}_p$ into a special subset of $C^1(\mbb{Z}_p)^*$, the collection of linear functionals on the set of continuously differentiable functions from $\mbb{Z}_p \to \mbb{C}_p$. We apply the results to the problem of interpolating Gauss sums attached to Dirichlet characters and the explicit annihilation of real ideal classes.
\end{abstract}

\maketitle

\section{Introduction}

Fix an odd prime $p$ and let $m$ be a positive integer co-prime to $p$. For an integer $n$, we let $\zeta_n=e^{2\pi i/n}$ so that $\zeta_n^d = \zeta_{n/d}$ for every $d \mid n$. Let $K_n = \mbb{Q}(\zeta_{mp^{n+1}})$, and let $G_n = \Gal(K_n/K_0)$.

We take a moment to review some classical theory from which this paper draws inspiration. Let $\theta_n \in \mbb{Q}[\Gal(K_n/\mbb{Q})]$ denote the classical Stickelberger element attached to the number field $K_n$. Recall that $\theta_n$, once properly made integral, annihilates the class group of $K_n$. Suppose $\varphi$ is a non-trivial even Dirichlet character of conductor $mp^{n+1}$ taking values in $K$, a finite extension of $\mbb{Q}_p$. The character $\varphi$ decomposes uniquely into a product of a tame character $\chi$ and a wild character $\psi$. Let $\theta_n(\chi) \in K[G_n]$ denote the $\chi$-part of $\theta_n$. In a celebrated work \cite{I}, Iwasawa showed that the sequence $(\theta_n(\chi)) \in \varprojlim K[G_n]$ (the projective limit taken with respect to the natural maps) is associated in a natural way to a function $F_{\chi}(T) \in \mfrak{o} \llbracket T-1 \rrbracket$ where $\mfrak{o}$ is the integer ring of $K$. What's more, this function is essentially the $p$-adic $L$-function of Leopoldt and Kubota. In fact, we have
\[ L_p(s,\chi\psi) = F_{\chi} \left( \zeta_{\psi} (1+p)^s \right)\]
where $\zeta_{\psi} = \overline{\psi}(1+p)$.

Unfortunately, if one restricts the action of $\theta_n$ to $K_n^+$, the maximal real subfield of $K_n$, it reduces to a multiple of the norm. With $\log_p$ denoting the Iwasawa logarithm, non-trivial explicit elements such as
\[ \vartheta_n = \sum_{\sigma \in G(k_n/\mbb{Q})} \log_p(1-\zeta_{mp^{n+1}}^\sigma) \sigma^{-1} \] were shown in \cite{A}, once properly made integral, to annihilate $\Cl(K_n^+) \otimes_{\mbb{Z}} \mcal{O}$ where $\mcal{O}$ is the ring of integers of the topological closure of $K_n^+ \hookrightarrow \mbb{Q}_p^{\alg}$. This article was born out of considering what analytic functions were naturally associate to the sequences $(\vartheta_n(\chi)) \in \varprojlim \mbb{Q}_p^{\alg}[G_n]$ (or more generally, to elements in $\varprojlim K_n^{\times}$) in analogy with Iwasawa's construction of $p$-adic $L$-functions.

Towards that end, let $k$ be any number field, and let
\[ k=k_0 \subset k_1 \subset k_2 \subset \cdots \subset \bigcup_{n=0}^{\infty} k_n = k_{\infty} \]
denote a $\mbb{Z}_p$-extension of $k$. So $\Gamma:=\Gal(k_{\infty}/k)$ is topologically isomorphic to $\mbb{Z}_p$, and $\Gamma_n = \Gal(k_n/k_0) \simeq \Gamma/\Gamma^{p^n}$. Let $\gamma_0$ be a fixed topological generator for $\Gamma$ and associate $\Gamma$ with $\mbb{Z}_p$ via the isomorphism $\gamma_0^a \mapsto a$. Let $\mfrak{p}$ be a prime of $k$ such that the inertia subgroup of $\mfrak{p}$ is $\Gal(k_{\infty}/k_i)$ for some $i$. This necessitates $\mfrak{p} \mid p$. The valuation $v_{\mfrak{p}}$ extends to $k_{\infty}$, and we let $\mcal{K}_n$ denote the completion of $k_n$ with respect to this valuation.

Let $\mbb{C}_p$ denote the topological closure of $\mbb{Q}_p^{\alg}$. Suppose $\mu = \{ \mu_n: \Gamma_n \to \mbb{C}_p\}_{n=0}^{\infty}$ is a collection of maps with the following property:
\[ \mu_n(x) = \sum_{y \mapsto x} \mu_{n+1}(x) \]
where $\Gamma_{n+1} \to \Gamma_n$ naturally. We call such a collection of maps a \textit{distribution} on $\Gamma$. We denote the ring (under convolution) of all distributions on $\Gamma$ by $\mcal{D}(\Gamma)$, and we write $\mu(a+p^n\mbb{Z}_p)$ in place of the more cumbersome $\mu_n(\gamma_0^a \bmod{\Gamma^{p^n}})$.

Note that the $\Gamma$-map $\mcal{D}(\Gamma) \to \varprojlim \mbb{C}_p[\Gamma_n]$ defined by
\[ \mu \mapsto \left( \sum_{a=0}^{p^n-1} \mu(a + p^n\mbb{Z}_p) \gamma_0^{-a} \right) \]
prescribes an isomorphism of rings. So the elements $(\vartheta_n(\chi)) \in \varprojlim \mbb{Q}_p^{\alg}[G_n]$ naturally give rise to distributions in $\mcal{D}(\Gamma)$ through the inverse of this map. On the other hand, $(\vartheta_n(\chi)) \in \varprojlim \mbb{Q}_p^{\alg}[G_n]$ as a byproduct of $(1-\zeta_{mp^{n+1}}) \in \varprojlim K_n^{\times}$, the projective limit taken with respect to the norm maps. Taken together, this uncovers a very natural source for distributions: define $\varprojlim k_n^{\times} \to \mcal{D}(\Gamma)$ by $(\ell_n) \mapsto \lambda$ where
\[ \lambda(a + p^n \mbb{Z}_p) = -\log_p \big( \ell_n^{\gamma_0^a} \big) \in \mcal{K}_n. \footnote{Our choice of sign reflects the formula for $L_p(1,\chi)$.}\]
Let $\mcal{M}(\Gamma)$ denote the collection of $\mbb{Z}_p$-valued distributions, and let $\mcal{K}(\Gamma)$ denote the $\mcal{M}(\Gamma)$-module generated by the image of the map described above.

What does one do with distributions anyway? For $\mu \in \mcal{D}(\Gamma)$, we say that a function $f: \mbb{Z}_p \to \mbb{C}_p$ is \emph{$\mu$-integrable} to mean that the limit
\[ \int_{\mbb{Z}_p} f(x)\ d\mu :=\lim_{n \to \infty} \sum_{a=0}^{p^n-1} f(a) \mu(a+p^n \mbb{Z}_p) \]
exists. We call this limit the \emph{Volkenborn integral of $f$ with respect to $\mu$}. The distinguishing feature of Volkenborn integration is the uniform choice of representatives from the classes $a+ p^n \mbb{Z}_p$ where $0 \leq a < p^n-1$ (namely, the choosing of $a$ itself).

Thus distributions give rise to linear functionals on appropriate function spaces. For example, it's well known that for every $\mu \in \mcal{M}(\Gamma)$, the collection $C(\mbb{Z}_p)$ of continuous functions on $\mbb{Z}_p$ are $\mu$-integrable. So every $\mu \in \mcal{M}(\Gamma)$ determines a linear functional on $C(\mbb{Z}_p)$ where
\[ \mu(f) := \int_{\mbb{Z}_p} f\ \mathrm{d}\mu.\]
What's more, the Fourier transform $\mcal{M}(\Gamma) \to \Lambda:=\mbb{Z}_p\llbracket T - 1 \rrbracket$ given by $\mu \mapsto  \widehat{\mu}(T)$ where
\[ \widehat{\mu}(T)=\mu(T^x) = \int_{\mbb{Z}_p} T^x\ \mathrm{d}\mu(x) = \sum_{m=0}^{\infty} \left( \int_{\mbb{Z}_p} \binom{x}{m}\ \mathrm{d}\mu(x) \right) (T-1)^m \]
is a well-defined isomorphism. All told we have natural isomorphisms
\begin{center}
\begin{tikzpicture}[scale=.75]
\node (M) at (-1,0) {$\mcal{M}(\Gamma)$};
\node (L) at (1,0) {$\Lambda$};
\node (GR) at (1,-2) {$\mbb{Z}_p\llbracket \Gamma \rrbracket$};

\node (m) at (6,0) {$\mu$};
\node (Fm) at (8,0) {$F_{\mu}$};
\node (gr) at (8,-2) {$\left( \sum \mu(a+p^n\mbb{Z}_p) \gamma_0^{-a} \right)$};

\node (text) at (3.5,-1) {given by};

\draw [->,-latex] (M) to node {} (L);
\draw [->,-latex] (L) to node {} (GR);
\draw [->,-latex] (M) to node {} (GR);

\draw [|->] (m) to node {} (Fm);
\draw [|->] (Fm) to node {} (gr);
\draw [|->] (m) to node {} (gr);
\end{tikzpicture}
\end{center}
If $M$ is a module over $\mcal{M}(\Gamma)$ or $\mbb{Z}_p\llbracket \Gamma \rrbracket$ naturally, then we consider it a module over $\Lambda$ (or any of the others for that matter) through the above diagram. In particular, extend the Iwasawa logarithm $\log_p$ to a function $\Log_p: k_n^{\times} \otimes_{\mbb{Z}} \mbb{Z}_p \to \mbb{C}_p$ in the obvious way: $\Log_p( \ell \otimes x) = x \log_p(\ell)$. Then it's straightforward to verify that the map $\varprojlim k_n^{\times} \otimes_{\mbb{Z}} \mbb{Z}_p \to \mcal{K}(\Gamma)$ defined by
\[ \mfrak{l}_n \mapsto \left( \mfrak{L} : a+ p^n\ZP \mapsto \Log_p\big( \mfrak{l}_n^{\gamma_0^a} \big) \right) \]
is a $\Lambda$-morphism.

Our main result is that continuously differentiable functions are $\lambda$-integrable for every $\lambda \in \mcal{K}(\Gamma)$, in other words
\begin{thm}
Let $\lambda \in \mcal{K}(\Gamma)$. Then $\lambda$ is a linear functional on $C^1(\mbb{Z}_p)$ where
\[ \lambda(f) = \int_{\mbb{Z}_p} f\ \mathrm{d}\lambda.\]
\end{thm}

In particular, the Fourier transform $\widehat{\lambda}(T) \in \mbb{C}_p\llbracket T-1 \rrbracket$ exists and has radius of convergence $ \geq 1$. The analytic functions $\widehat{\lambda}(T)$ are like $L$-functions for the underlying norm coherent sequence. For example, consider the following special case. Suppose $k/\mbb{Q}$ is an abelian number field whose conductor is not divisible by $p^2$, and let $F$ be any abelian number field linearly disjoint from $k$ and of conductor co-prime to $p$. If $k_{\infty}/k$ is the cyclotomic $\mbb{Z}_p$-extension of $k$, then the tower of number fields $Fk_n$ forms the cyclotomic $\mbb{Z}_p$-extension of $Fk$, and we consider $\Gamma_n$ (resp. $\Delta:=\Gal(k_0/\mbb{Q})$) as being contained in (resp. a quotient of) the set of automorphisms of $\Gal(Fk_n/\mbb{Q})$ fixing $F$. For a character $\chi$ of $\Delta$, define $\varprojlim (Fk_n)^{\times} \to \mcal{D}(\Gamma)$ by $(\ell_n) \mapsto \lambda_{\chi}$ where
\[ \lambda_{\chi}(a+p^n\mbb{Z}_p) = -\sum_{\delta \in \Delta} \log_p(\ell_n^{\gamma_0^a \delta}) \overline{\chi}(\delta).\]
Let $\mcal{K}_{\chi}^F(\Gamma)$ denote the $\Lambda$-module generated by the image of the map described above. The functions $\widehat{\lambda}_{\chi}(T)$ (or $\widehat{\lambda}(T)$, for that matter) interpolate values reminiscent of those found in the formula for $L_p(1,\varphi)$, the $p$-adic $L$-function of Leopoldt, Kubota, Iwasawa, et al. As a straightforward consequence of the above theorem, we have
\begin{thm}
Let $\lambda_{\chi} \in \mcal{K}_{\chi}^F(\Gamma)$. Then $\lambda_{\chi}$ is a linear functional on $C^1(\mbb{Z}_p)$ where
\[ \lambda_{\chi}(f) = \int_{\mbb{Z}_p} f\ \mathrm{d}\lambda_{\chi}.\]
If $\psi$ is a character of $\Gamma_n$ with $\zeta_{\psi} = \overline{\psi}(\gamma_0)$ and $(\ell_n) \mapsto \lambda_{\chi}$, then
\[ \widehat{\lambda}_{\chi}(\zeta_{\psi}) = -\sum_{\sigma} \log_p\big( \ell_n^{\sigma} \big) \overline{\varphi}(\sigma) \]
where the sum runs over all $\sigma \in \Gamma_n \times \Delta = \Gal(k_n/\mbb{Q})$ and $\varphi = \chi \psi$.
\end{thm}

We apply the above results to the problem of interpolating Gauss sums attached to a Dirichlet character. Particularly interesting is the case when the tamely ramified character $\chi$ is of conductor $p$. In this case, the Gauss sums
\[ \tau(\chi \psi) = \sum_{a=1}^{p^{n+1}} \chi\psi(a) \zeta_{p^{n+1}}^a \]
are essentially interpolated from the Fourier transform of $\lambda_{\chi} \in \mcal{K}_{\chi}^{\mbb{Q}(\zeta_{p-1})}(\Gamma)$ where the underlying norm coherent sequence generates the projective limit of principal units of $\mbb{Q}_p(\zeta_{p^{n+1}})$. Since it's peripheral to the interpolation problem, we also show how to use the special values of the functions $\widehat{\lambda_{\chi}}(T)$ to construct an explicit sequence $(\vartheta_n) \in \mbb{Z}_p[\Gamma_n]$ such that $\vartheta_n$ annihilates the $\chi$-part of the Sylow $p$-subgroup of $\Cl \big( \mbb{Q}(\zeta_{p^{n+1}}^+) \big)$ for every $n \geq 0$.

\section{Volkenborn Distributions}

In this section we give an overview of the theory of \emph{Volkenborn distributions} of which  distributions in $\mcal{K}(\Gamma)$ are a special case. C. Barbacioru \cite{B} developed the Volkenborn distribution in his doctoral dissertation. This section is largely an overview of the tools from \cite{B} that will be needed in the sequel.

\begin{definition}\label{bd1}
A distribution $\mu$, on $\ZP$, is said to be Volkenborn if there exists $B(\mu)\in\mathbb{R}_{\geq 0}$ such that
\begin{equation*}
|p\mu(a+p^{n+1}\ZP)-\mu(a+p^n\ZP)|_p\leq B(\mu)
\end{equation*}
for all $a\in\ZP$ and $n\in\mathbb{Z}_{\geq 0}$
\end{definition}
Note that all $p$-adically bounded distributions are necessarily Volkenborn, but a distribution need not be bounded to be Volkenborn. In fact, the prototype Volkenborn distribution is the Haar distribution: $a+p^n \mbb{Z}_p \mapsto \frac{1}{p^n}$.

\begin{lemma}\label{bl1}
Let $\mu$ be a Volkenborn distribution and let $f_n: \mbb{Z}_p \to \mbb{C}_p$ be defined by $f_n: x \mapsto p^n \mu(x + p^n\ZP)$. Then there exists a continuous and bounded function $f: \ZP \to \mbb{C}_p$ such that $f_n \rightrightarrows f$ uniformly on $\ZP$.
\end{lemma}
\begin{proof}
Note that
\[ p^n\mu(a+p^n\ZP) = \left( \sum_{j=1}^n p^{j-1} ( p\mu(a+p^j\ZP)-\mu(a+p^{j-1}\ZP)) \right)+\mu(\ZP).\]
The terms of the sum go to zero as $j \to \infty$ since $\mu$ is Volkenborn. It follows that the sum converges. Define $f: \mbb{Z}_p \to \mbb{C}_p$ by $x \mapsto \lim p^n \mu(x + p^n\ZP)$. Note the above shows that $f$ is bounded, in fact, $|f(x)| \leq \max \{ B(\mu), \mu(\ZP)\}$ for all $x \in \ZP$.

Now, let $x \in \mbb{Z}_p$ be arbitrary. Let $m>n$ be sufficiently large so that
\begin{align*}
\left| f(x) - f_n(x) \right|_p &\leq \max \left\{ \{ \left| f_{j+1}(x) - f_j(x)\right|_p\}_{j=n}^{m-1} \cup \{ \left| f(x) - f_m(x) \right|_p\} \right\} \\
&\leq \max \{ \left| f_{j+1}(x) - f_j(x)\right|_p\}_{j=n}^{m-1} \\
&\leq \frac{B(\mu)}{p^n}.
\end{align*}
The above bound does not depend on $x$, so $f_n \rightrightarrows f$. The function $f$ is continuous since it is a uniform limit of continuous functions on a compact set.
\end{proof}

For a Volkenborn distribution $\mu$, we want to show that all $C^1$ functions are $\mu$-integrable.  The strategy will be to first show that polynomials are $\mu$-integrable. This, in conjunction with properties of Mahler series of $C^1$ functions, will give us the $\mu$-integrability of $C^1$ functions.

\begin{proposition}\label{bp1}
Let $\mu$ be a Volkenborn distribution and $P$ be a polynomial.  Then $P$ is $\mu$-integrable.
\end{proposition}
\begin{proof}
Since limits are finitely additive, it suffices to show that $P(x)=x^m$ is $\mu$-integrable for all $m\in\mathbb{Z}_{\geq 0}$.  We proceed by induction. For $P(x) =1$, we have
\[ \int_{\mbb{Z}_p} \mathrm{d}\mu = \lim_{n \to \infty} \sum_{a=0}^{p^n-1} \mu(a+p^n\mbb{Z}_p) = \mu(\ZP).\]
Now, let $S_{n,m}:=\sum_{j=0}^{p^n-1}j^m\mu(j+p^n\ZP)$ for $m,n\in\mathbb{Z}_{\geq 0}$.  We wish to show that for a fixed $m\geq 1$ that $S_{n,m}$ is a Cauchy sequence. Note that

\begin{align*}
S_{n+1,m}-S_{n,m}&=\sum_{j=0}^{p^n-1}\sum_{k=0}^{p-1}((j+kp^n)^m-j^m)\mu(j+kp^n+p^{n+1}\ZP)\\
&=\sum_{j=0}^{p^n-1}\sum_{k=0}^{p-1}\sum_{l=1}^m\binom{m}{l}(kp^n)^l j^{m-l}\mu(j+kp^n+p^{n+1}\ZP).\hspace{2ex}\tag{$\star$}
\end{align*}
By Lemma \ref{bl1} we only need to show that the $l=1$ term from ($\star$) is small.  To do so, we will rewrite that term as follows:
\begin{equation*}
\sum_{j=0}^{p^n-1}\sum_{k=0}^{p-1}mkp^nj^{m-1}\mu(j+kp^n+p^{n+1}\ZP)=a_n+b_n
\end{equation*}
where
\begin{align*}
a_n &=\sum_{j=0}^{p^n-1}\sum_{k=0}^{p-1}mkp^nj^{m-1}(\mu(j+kp^n+p^{n+1}\ZP)-\frac{1}{p}\mu(j+p^n\ZP)), \\
b_n &=\sum_{j=0}^{p^n-1}\sum_{k=0}^{p-1}mkp^{n-1}j^{m-1}\mu(j+p^n\ZP).
\end{align*}
It remains to show that both $a_n$ and $b_n$ go to zero as $n\to\infty$. For $a_n$, we have
\begin{align*}
|a_n|_p&=\left|\sum_{j=0}^{p^n-1}\sum_{k=0}^{p-1}mkp^nj^{m-1}(\mu(j+kp^n+p^{n+1}\ZP)-\frac{1}{p}\mu(j+p^n\ZP))\right|_p\\
&\leq |mp^n(\mu(j+kp^n+p^{n+1}\ZP)-\frac{1}{p}\mu(j+p^n\ZP))|_p\\
&\leq p^{1-n}B(\mu).
\end{align*}
It follows that $a_n\to 0$ as $n\to\infty$. For $b_n$, we have
\[ b_n = \sum_{j=0}^{p^n-1}\sum_{k=0}^{p-1}mkp^{n-1}j^{m-1}\mu(j+p^n\ZP) = \frac{p-1}{2}mp^n S_{n,m-1}.\]
By the inductive hypothesis $\{S_{n,m-1}\}_{n=0}^\infty$ is a bounded sequence (since it is a convergent sequence).  It follows that $b_n\to 0$ as $n\to\infty$. This shows that $S_{n,m}$ is a Cauchy sequence, so $\lim_{n\to\infty} S_{n,m}$ converges.
\end{proof}

Since $C^1$ functions are determined by their Mahler series, it is important to know bounds on $\left|\int_{\ZP}\binom{x}{m}\mathrm{d}\mu(x)\right|$.  The next proposition gives such a bound.
\begin{proposition}\label{bp2}
Let $\mu$ be a Volkenborn distribution.  Then there exists $c\in\mathbb{R}_{\geq 0}$ such that for all $m\in\mathbb{Z}_{\geq 0}$ we have
\begin{equation*}
\left|\int_{\ZP}\binom{x}{m}\mathrm{d}\mu(x)\right|_p\leq cm
\end{equation*}
\end{proposition}
\begin{proof}
For $m=0$, we know that $\int_{\ZP}\mathrm{d}\mu(x)$ exists and equals $\mu(\ZP)$.  From this point on let $m\in\mathbb{Z}_{\geq 1}$.  By Proposition \ref{bp1} we know that $\binom{x}{m}$ is $\mu$ integrable.  The proof of the inequality proceeds in a similar manner to the proof of Proposition \ref{bp1}, and we will use the sequence $\{T_{n,m}\}_{n=0}^\infty$ where
\begin{equation*}
T_{n,m}:=\sum_{j=0}^{p^n-1}\binom{j}{m}\mu(j+p^n\ZP).
\end{equation*}
Note that
\begin{equation}
|T_{n+1,m}-T_{n,m}|_p = \left| \sum_{j=0}^{p^n-1} \sum_{k=0}^{p-1} \left( \binom{j+kp^n}{m} - \binom{j}{m} \right) \mu(j+kp^n+p^{n+1}\ZP) \right|_p. \label{eq3}
\end{equation}
To estimate Equation \ref{eq3}, we use the binomial identity
\begin{equation*}
\binom{j+kp^n}{m}=\sum_{l=0}^m\binom{j}{l}\binom{kp^n}{m-l}.
\end{equation*}
The right hand side of Equation \ref{eq3} becomes
\begin{equation}\label{eq4}
\left| \sum_{j=0}^{p^n-1} \sum_{k=0}^{p-1} \sum_{l=0}^{m-1} \binom{j}{l} \binom{kp^n}{m-l} \mu(j+kp^n+p^{n+1}\ZP) \right|_p.
\end{equation}
We can bound each term of the sum from Equation \ref{eq4} as follows:
\begin{align*}
\left| \binom{j}{l} \binom{kp^n}{m-l} \mu(j+kp^n+p^{n+1}\ZP) \right|_p & \leq \left| \binom{kp^n}{m-l} \mu(j+kp^n+p^{n+1}\ZP) \right|_p\\
&=\left|\frac{kp^n}{m-l}\binom{kp^n-1}{m-l-1}\mu(j+kp^n+p^{n+1}\ZP)\right|_p\\
&\leq\left|\frac{kp^n}{m-l}\mu(j+kp^n+p^{n+1}\ZP)\right|_p\\
&\leq p^{-n}m|\mu(j+kp^n+p^{n+1}\ZP)|_p\\
&\leq Cpm \qquad \text{by Lemma \ref{bl1}.}
\end{align*}
This estimate gives us that Equation \ref{eq4} is bounded above by $Cpm$.  In other words,
\begin{equation}
|T_{n+1,m}-T_{n,m}|_p\leq Cpm.
\end{equation}

Now we are in position to prove the result.
\begin{align*}
|T_{n,m}|_p&=\left|\sum_{j=0}^n (T_{n,m}-T_{n-1,m})+T_{0,m}\right|_p\\
&\leq \max\{ Cpm, |T_{0,m}|\}\\
&= \max\{ Cpm, |\mu(\ZP)|_p\}.
\end{align*}
Letting $c=\max\{Cp, |\mu(\ZP)|_p\}$, we see that $|T_{n,m}|_P\leq cm$.  This gives us that $|\int_{\ZP}\binom{x}{m}\mathrm{d}\mu(x)|_p\leq cm$, as claimed.
\end{proof}

It is important to note that $c$ from Proposition \ref{bp2} is independent of $m$.

\begin{theorem}\label{bt1}
Let $f\in C^1(\mbb{Z}_p)$ and $\mu$ be a Volkenborn distribution.  Then $f$ is $\mu$-integrable.
\end{theorem}
\begin{proof}
Since $f\in C^1$, we know that the Mahler series of $f$ is of the form
\[ \sum_{m=0}^\infty a_m\binom{\cdot}{m} \quad \text{where} \quad \lim_{m\to\infty}m|a_m|_p=0 \]
(see \cite{YA}).  We will show that
\begin{equation}\label{eq5}
\int_{\ZP}f(x)\mathrm{d}\mu(x)=\sum_{m=0}^\infty a_m\int_{\ZP}\binom{x}{m}\mathrm{d}\mu(x).
\end{equation}
By Proposition \ref{bp2} we know that $\left|\int_{\ZP}\binom{x}{m}\mathrm{d}\mu(x)\right|_p\leq cm$.  This tells us that
\begin{equation*}
\lim_{m\to\infty}a_m\int_{\ZP}\binom{x}{m}\mathrm{d}\mu(x)=0.
\end{equation*}
Thus the right hand side of Equation \ref{eq5} converges.

Now we will show that the left hand side of Equation \ref{eq5} exists and equals the right hand side of the same equation.  To do so we will use the sequence $\{T_{n,m}\}_{m=0}^\infty$ from Proposition \ref{bp2}.  The proof of Proposition \ref{bp2} showed that there exists $c\in\mathbb{R}_{\geq 0}$ such that $|T_{n,m}|_p\leq cm$.

Let $\epsilon>0$.  Then there exists $M\in\mathbb{Z}_{>0}$ such that for all $m\geq M$ we have that $|a_mT_{n,m}|_p<\epsilon$.  Also, there exists $N\in\mathbb{Z}_{>0}$ such that for all $0\leq m\leq M$ and $n\geq N$ we have that $|a_m(T_{n,m}-\int_{\ZP}\binom{x}{m}\mathrm{d}\mu(x))|_p<\epsilon$.

Let $n\geq N$.  Then
\begin{equation*}
\sum_{j=0}^{p^n-1}f(j)\mu(j+p^n\ZP)-\sum_{m=0}^\infty a_m\int_{\ZP}\binom{x}{m}\mathrm{d}\mu(x)=a_M+b_M
\end{equation*}
where
\begin{align*}
a_M&=\sum_{j=0}^{p^n-1}\sum_{m=0}^Ma_m\binom{j}{m}\mu(j+p^n\ZP)-\sum_{m=0}^Ma_m\int_{\ZP}\binom{x}{m}\mathrm{d}\mu(x)\\
&=\sum_{m=0}^Ma_m\left(T_{n,m}-\int_{\ZP}\binom{x}{m}\mathrm{d}\mu(x)\right)
\end{align*}
and
\begin{align*}
b_M&=\sum_{j=0}^{p^n-1}\sum_{m=M}^\infty a_m\binom{j}{m}\mu(j+p^n\ZP)-\sum_{m=M}^\infty a_m\int_{\ZP}\binom{x}{m}\mathrm{d}\mu(x)\\
&=\sum_{m=M}^\infty a_mT_{n,m}-\sum_{m=M}^\infty a_m \int_{\ZP}\binom{x}{m}\mathrm{d}\mu(x).
\end{align*}
We have $|a_M|_p<\epsilon$ by our choice of $n$ (which depends on $M$), and $|b_M|_p<\epsilon$ by our choice of $M$. It follows that $\int_{\ZP}f(x)\mathrm{d}\mu(x)$ exists, so $f$ is $\mu$-integrable.
\end{proof}

\section{The module of Volkenborn Distributions}

Let $\mcal{V}(\Gamma)$ denote the collection of Volkenborn distributions. Recall that $\Lambda \simeq \mcal{M}(\Gamma)$ acts on $\mcal{V}(\Gamma)$ by convolution.

\begin{lemma} \label{vmod}
$\mcal{V}(\Gamma)$ is a $\Lambda$-module.
\end{lemma}
\begin{proof}
Let $\nu \in \mcal{V}(\Gamma)$ and $\mu$ a bounded distribution, i.e., a distribution such that there exists $B \in \mbb{R}_{\geq 0}$ satisfying
\[ |\mu(a+p^n\ZP) |_p \leq B \]
for all $a$ and $n$. We show more generally that $\nu * \mu \in \mcal{V}(\Gamma)$. By the definition for convolution, we have
\[ (\nu*\mu)(a + p^{n+1}\ZP) = \sum_{j=0}^{p^n-1} \sum_{k=0}^{p-1} \nu(j+kp^n + p^{n+1}\ZP)\mu(a-j-kp^n+p^{n+1}\ZP),\]
and similarly
\[ (\nu*\mu)(a+p^n\ZP) = \sum_{j=0}^{p^n-1} \nu (j+p^n\ZP) \sum_{k=0}^{p-1} \mu(a-j-kp^n+p^{n+1}\ZP).\]
So we see that $p \cdot (\nu*\mu)(a+ p^{n+1}\ZP) - (\nu*\mu)(a+p^n\ZP)$ equals
\[ \sum_{j=0}^{p^n-1} \sum_{k=0}^{p-1} \mu(a-j-kp^n+p^{n+1}\ZP) \cdot \left( p \cdot \nu(j+kp^n + p^{n+1}\ZP) - \nu(j+p^n\ZP) \right).\]
Since $\mu$ is bounded and
\[p \cdot \nu(j+kp^n + p^{n+1}\ZP) - \nu(j+p^n\ZP) \]
is bounded independently from $j$ and $k$, we see that $\nu*\mu \in \mcal{V}(\Gamma)$.
\end{proof}

The Fourier transform of a Volkenborn distribution is guaranteed to exist from \Cref{bt1}. We now study how convolution by $\mu \in \mcal{M}(\Gamma)$ affects the Fourier transform of $\nu \in \mcal{V}(\Gamma)$. For a Volkenborn distribution $\nu$, let $f_{\nu}$ denote the function defined by $x \mapsto \lim p^n \nu(x + p^n\ZP)$. Recall that $f_{\nu}$ is a bounded continuous function by \Cref{bl1}. Let $\mathbf{S}$ denote the indefinite-sum operator. For $f \in C(\ZP)$, the action of $\mathbf{S}$ on $f$ simply shifts the Mahler expansion in the following way:
\[ \mathbf{S} f = \mathbf{S} \sum_{m=0}^{\infty} \binom{\cdot}{m} (\nabla^m f)(0) = \sum_{m=0}^{\infty} \binom{\cdot}{m+1} (\nabla^m f)(0) \in C(\ZP)\]
where $(\nabla f)(x) = f(x+1) - f(x)$ is the finite-difference operator. The reader should consult \cite[Chapter V]{R} for more details.
\begin{proposition} \label{transform}
Let $\mu \in \mcal{M}(\Gamma)$. For every $\nu \in \mcal{V}(\Gamma)$, we have
\[ (\widehat{\nu *\mu})(T) =  \widehat{\nu}(T) \cdot \widehat{\mu}(T) -\log_p(T) \cdot \sum_{m=0}^{\infty} \mu\left( \mathbf{S}^{m+1} (f_{\nu} \circ \iota) \right) (T-1)^m \]
where $\iota:x \mapsto -1-x$ is the canonical involution of $\ZP$.
\end{proposition}
\begin{proof}
Note that
\[ \sum_{a,b=0}^{p^n-1} T^{a+b} \nu(a + p^n\mbb{Z}_p) \mu(b + p^n\mbb{Z}_p) \xrightarrow{n \to \infty} \int_{\mbb{Z}_p} T^x\ \mathrm{d}\nu(x) \cdot \int_{\mbb{Z}_p} T^x\ \mathrm{d}\mu(x).\]
Consider the sum on the left. Collecting all terms such that $a+b \equiv c \bmod{p^n}$, we see that it equals
\[ \sum_{c=0}^{p^n-1} T^c (\mu*\nu)(c+p^n\mbb{Z}_p) + \sum_{c=0}^{p^n-2} \sum_{d=1}^{p^n-c-1} (T^{c+p^n} - T^c) \mu(c+d + p^n\mbb{Z}_p) \nu(-d + p^n\mbb{Z}_p).\]
As $n \to \infty$, the term on the left converges to $(\widehat{\mu*\nu})(T)$ since $\mu* \nu$ is Volkenborn. Hence the term on the right converges. We rewrite that term as
\[ \frac{T^{p^n} -1}{p^n} \sum_{m=0}^{\infty} \left( \sum_{c=0}^{p^n-2} \binom{c}{m} \sum_{d=1}^{p^n-c-1} \mu(c+d+p^n\mbb{Z}_p) \cdot p^n \nu(-d+p^n\mbb{Z}_p) \right) (T-1)^m.\]
This expression converges to $\log_p(T) \cdot G(T)$ where the $m$-th coefficient of $G(T)$ equals
\[ g_m:=\lim_{n \to \infty} \sum_{c=0}^{p^n-2} \binom{c}{m} \sum_{d=1}^{p^n-c-1} \mu(c + d + p^n \mbb{Z}_p) \cdot p^n \nu(-d+p^n\mbb{Z}_p).\]
We collect terms according to $\mu(j+p^n\ZP)$ obtaining
\[ g_m= \lim_{n \to \infty} \sum_{j=0}^{p^n-1} \left( \binom{\cdot}{m} \varoast (f_n \circ \iota) \right) (j) \cdot  \mu(j+p^n\ZP)\]
where $f_n : x \mapsto p^n \nu(x + p^n \ZP)$ and $\varoast$ is the shifted-convolution product. We now use the fact that $\binom{\cdot}{m} \varoast g = \mathbf{S}^{m+1} g$ and $f_n \rightrightarrows f_{\nu}$ to obtain
\begin{align*}
g_m &= \lim_{n \to \infty} \sum_{j=0}^{p^n-1} \mathbf{S}^{m+1} (f_{\nu} \circ \iota)(j) \mu(j+p^n\ZP) \\
&= \int_{\ZP} \mathbf{S}^{m+1} (f_{\nu} \circ \iota) \ \mathrm{d}\mu(x).
\end{align*}
This completes the proof of the proposition.
\end{proof}

\begin{remark} Note that if $\nu \in \mcal{M}(\Gamma)$, then $f_{\nu} \equiv 0$. So we recover the well-known fact that $\widehat{\nu * \mu}=\widehat{\nu} * \widehat{\mu}$, the Fourier transform of a convolution of measures equals the convolution of Fourier transforms.
\end{remark}

We now show that $\mcal{V}(\Gamma)$ is populated by members of $\mcal{K}(\Gamma)$.
\begin{theorem}
$\mcal{K}(\Gamma)$ is a sub-module of $\mcal{V}(\Gamma)$.
\end{theorem}
\begin{proof}
In light of \Cref{vmod}, it suffices to prove that the generators of $\mcal{K}(\Gamma)$ reside in $\mcal{V}(\Gamma)$. Let $(\ell_n) \in \varprojlim k_n^{\times}$ with associated distribution $\lambda$. We have
\[ p \lambda(a+p^{n}\mbb{Z}_p) - \lambda(a+p^{n-1}\mbb{Z}_p) = \log_p \left( \frac{\ell_{n-1}^{\gamma_0^a}}{\ell_{n}^{p\gamma_0^a} } \right). \]
Observe that
\[ \frac{\ell_{n-1}^{\gamma_0^a}}{\ell_{n}^{p\gamma_0^a} } \xrightarrow{N^{n}_{n-1}} 1 \]
where $N^{n}_{n-1}$ is the norm from $k_n$ to $k_{n-1}$. Since $k_n/k_{n-1}$ is a cyclic extension, Hilbert's Theorem 90 gives an element $\alpha_{n} \in k_n^{\times}$ such that
\[ \frac{\ell_{n-1}^{\gamma_0^a}}{\ell_{n}^{p\gamma_0^a} } = \alpha_{n}^{\gamma_0^a(\gamma_{n}-1)} \quad \text{where} \quad \gamma_{n} = \gamma_0^{p^{n-1}}.\]
It remains to show that $\log_p \big( \alpha_n^{\gamma_0^a(\gamma_n-1)} \big)$ is bounded independent of $a$ and $n$. In fact, we need only show that it is bounded independent of $a$ and $n$ for all $n$ sufficiently large.

Assume that the inertia subgroup for $\mfrak{p}$ of $k$ is $\Gal(k_{\infty}/k_i)$ and let $n \geq i$. Let $\pi_{n}$ be a local parameter for $\mcal{K}_{n}$. Since $\mcal{K}_{n}/\mcal{K}_i$ is totally ramified, it follows that
\[ N^{n}_{i}(\pi_{n}) = \pi_i \]
is a local parameter for $\mcal{K}_i$. Moreover, we get that
\[ N^{n}_i \left( \mcal{K}_{n}^{\times} \right) = \langle \pi_i \rangle \times N^{n}_i(U_{n}) \]
where $U_{n}$ denotes the units of $\mcal{K}_{n}$. Note that
\[ [U_i: N^{n}_i(U_{n})] = p^{n-i} \]
since $\mcal{K}_n/\mcal{K}_i$ is cyclic and totally ramified. Let $U_i^{(j)}$ denote the $j$-th group of principal units of $\mcal{K}_i$, so $U_i^{(j)} = 1+(\pi_i)^j \subset U_i$. Let $q$ denote the order of the residue class field for $\mcal{K}_i$, and recall the filtration
\[ U_i \supset U_i^{(1)} \supset U_i^{(2)} \supset \cdots \]
where
\[ [U_i^{(j)}: U_i^{(j+1)} ] =\begin{cases}
    q-1 & j=0 \\
    q & \text{else}.
    \end{cases} \]
Let $m$ be the smallest integer such that $U_i^{(m)} \subset N_i^{n}(U_{n})$, so $\langle \pi_i \rangle \times U_i^{(m)} \subseteq N^{n}_i(\mcal{K}_{n}^{\times})$. From the above filtration, we see that as $n$ increases so must $m$. Let $n$ be large enough so that $m >1$.

Since $\langle \pi_i \rangle \times U_i^{(m)} \subseteq N^n_i(\mcal{K}_n^{\times})$, local class field theory gives us that $\mcal{K}_n \subseteq \mcal{L}_m$ where $\mcal{L}_m$ is the field of $\pi_i^m$-division points of some Lubin-Tate module for $\pi_i$ (see \cite{LT,JN}). For a real number $s \geq -1$, we define the $s$-th ramification group
\[ G_s(\mcal{L}_m/\mcal{K}_i) = \{ \sigma \in \Gal(\mcal{L}_m/\mcal{K}_i) : w(\sigma(a)-a) \geq s+1 \quad \forall a \in \mcal{O} \}\]
where $\mcal{O}$ is the valuation ring of $\mcal{L}_m$ and $w$ is the valuation associate to its maximal ideal. The Lubin-Tate extensions have the property that
\[ G_{q^{m-1}-1}(\mcal{L}_m/\mcal{K}_i) = \Gal(\mcal{L}_m/\mcal{L}_{m-1}).\]
Let $H \subset \Gal(\mcal{L}_m/\mcal{K}_i)$ such that $\mcal{K}_n$ is the fixed field of $H$. A theorem of Herbrand (see \cite[II.10.7]{JN}) gives us that
\[ G_s(\mcal{L}_m/\mcal{K}_i) H /H = G_t(\mcal{K}_n/\mcal{K}_i) \quad \text{where} \quad t = \int_0^s \frac{dx}{[G_0(\mcal{L}_m/\mcal{K}_n): G_x(\mcal{L}_m/\mcal{K}_n)]}.\]
By the minimality of $m$, we have that
\[ G_{q^{m-1}-1}(\mcal{L}_m/\mcal{K}_i) = \Gal(\mcal{L}_m/\mcal{L}_{m-1}) \not \subseteq H,\]
so for $s=q^{m-1}-1$, we have $G_t(\mcal{K}_n/\mcal{K}_i)$ is non-trivial. We now obtain a crude but functional lower bound for the value $t$. Since $\mcal{L}_m/\mcal{K}_i$ is totally ramified, we have
\[ t = \frac{[\mcal{K}_n:\mcal{K}_i]}{[\mcal{L}_m:\mcal{K}_i]} \sum_{j=1}^{q^{m-1}-1} \# G_j(\mcal{L}_m/\mcal{K}_n) \geq \frac{p^{n-i}}{q-1} \cdot \frac{q^{m-1}-1}{q^{m-1}} \geq \frac{p^{n-i-1}}{q-1} =t(n)\]
where the last inequality follows because $m >1$. It follows that
\[ \gamma_n \in \Gal(\mcal{K}_n/\mcal{K}_{n-1}) \subseteq G_t(\mcal{K}_n/\mcal{K}_i) \subseteq G_{t(n)}(\mcal{K}_n/\mcal{K}_i).\]
Let $e(\pi_n:p)$ denote the ramification index of $\pi_n$ over $p$, then for all $n$ sufficiently large
\[ \alpha_n^{\gamma_0^a(\gamma_n-1)} - 1 \in (\pi_n)^{t(n)} \Rightarrow v_p\left( \alpha_n^{\gamma_0^a(\gamma_n-1)} - 1\right) \geq \frac{t(n)}{e(\pi_n:p)} = \frac{1}{p(q-1)e(\pi:p)}. \]
Whence $\log_p(\alpha_n^{\gamma_0^a(\gamma_n-1)})$ is bounded independent of $n$ and $a$ for all $n$ sufficiently large. This proves the theorem.
\end{proof}

\begin{remark} Let $H^1(\mbb{C}_p)$ denote the ring of power series in $\mbb{C}_p\llbracket T-1 \rrbracket$ convergent on the open ball of radius $1$ centered about $1$. This section shows that the map $\mathscr{F}:\varprojlim k_n^{\times} \otimes_{\mbb{Z}} \mbb{Z}_p \to H^1(\mbb{C}_p)/(\log_p(T))$ defined by
\[ (\mfrak{l}_n) \mapsto \widehat{\mfrak{L}} \bmod{(\log_p(T))}\]
is a $\Lambda$-morphism. If $(\mfrak{l}_n) \in \ker \mathscr{F}$, then for every $n \geq 0$, for every character $\psi$ of $\Gamma_n$, we have
\[ 0= \widehat{\mfrak{L}}(\zeta_{\psi}) = \sum_{a=0}^{p^n-1} \overline{\psi}(\gamma_0^a) \Log_p(\mfrak{l}_n^{\gamma_0^a}) \]
where
\[ e_{\psi} \cdot \sum_{a=0}^{p^n-1} \Log_p(\mfrak{l}_n^{\gamma_0^a}) \gamma_0^{-a} = \sum_{a=0}^{p^n-1} \overline{\psi}(\gamma_0^a) \Log_p(\mfrak{l}_n^{\gamma_0^a}) \cdot e_{\psi} \in \mbb{C}_p[\Gamma_n]\]
and $e_{\psi} \in \mbb{C}_p[\Gamma_n]$ is the idempotent associate to $\psi$. Since $\mbb{C}_p[\Gamma_n] = \bigoplus_{\psi} \mbb{C}_p e_{\psi}$, it follows that
\begin{align*}
\widehat{\mfrak{L}} \equiv 0 \bmod{(\log_p(T))} &\Leftrightarrow 0=\sum_{a=0}^{p^n-1} \Log_p(\mfrak{l}_n^{\gamma_0^a}) \gamma_0^{-a},\quad  \forall n\geq 0 \\
& \Leftrightarrow \mfrak{L}=0.
\end{align*}
Whether $\mfrak{L}$ is the $0$-distribution is a more delicate question. For suppose $\mfrak{l}_n = \sum (\ell_j \otimes x_j)$, then
\[ \mfrak{L}(p^n \ZP) = \Log_p(\mfrak{l}_n) = \sum x_j \log_p(\ell_j).\]
We now need to know whether the terms $\log_p(\ell_j)$ are $p$-adically independent, a question related to Leopoldt's conjecture.
\end{remark}

\section{Applications}

In this section we apply the above results to the problem of interpolating Gauss sums. For a Dirichlet character $\varphi$, let $f_{\varphi}$ denote the conductor of $\varphi$ and let $\tau(\varphi)$ denote the Gauss sum
\[ \tau(\varphi) = \sum_{a=1}^{f_{\varphi}} \varphi(a) \zeta_{f_{\varphi}}^a.\]
We associate Dirichlet characters of conductor dividing $mp^{n+1}$ to characters of $\Gal(\mbb{Q}(\zeta_{mp^{n+1}}/\mbb{Q})$ in the obvious way: $\varphi(\sigma_a) = \varphi(a)$ where $\sigma_a: \zeta_{mp^{n+1}} \mapsto \zeta_{mp^{n+1}}^a$.

\begin{theorem} \label{app}
Let $k = \mbb{Q}(\zeta_{mp})$ and $k_{\infty}/k$ the cyclotomic $\mbb{Z}_p$-extension of $k$. Let $F = \mbb{Q}(\zeta_{p-1})$. Let $\chi$ be a character of $\Delta$ with $m \mid f_{\chi}$, and let $\lambda_{\chi}$ be the distribution associate to the norm coherent sequence $(\zeta_{p-1}^t-\zeta_{mp^{n+1}}) \in \varprojlim (Fk_n)^{\times}$. If $\psi$ is a character of $\Gamma_n$ with $\zeta_{\psi} = \overline{\psi}(\gamma_0)$, then
\[ \widehat{\lambda_{\chi}}(\zeta_{\psi}) = -\big(1-\overline{\varphi}(p) \big) \sum_{a=1}^{f_{\varphi}} \log_p(\zeta_{p-1}^t-\zeta_{f_{\varphi}}^a) \overline{\varphi}(a)\]
where $\varphi = \chi\psi$.
\end{theorem}
\begin{proof}
Note that $\zeta_{p-1}^t - \zeta_{mp^{n+1}}$ is indeed a norm coherent sequence by virtue of our choice for $\zeta_n$, namely, $\zeta_n^d = \zeta_{n/d}$ for all $d \mid n$. So $\lambda_{\chi}$ is an honest distribution satisfying
\[ \widehat{\lambda_{\chi}}(\zeta_{\psi}) = \int_{\mbb{Z}_p} \overline{\psi}(\gamma_0)^x\ \mathrm{d}\lambda_{\chi}(x)= \sum_{a=0}^{p^n-1} \overline{\psi}(\gamma_0^a) \lambda_{\chi}(a+p^n \ZP) \]
where the last equality follows since $\psi$ is a character of $\Gamma_n$. Hence
\begin{align*}
\widehat{\lambda_{\chi}}(\zeta_{\psi}) &= -\sum_{\substack{b=1 \\ (b,mp)=1}}^{mp^{n+1}} \log_p(\zeta_{p-1}^t-\zeta_{mp^{n+1}}^b) \overline{\varphi}(b) \\
&= -\big(1 - \overline{\varphi}(p) \big) \sum_{b=1}^{mp^{n+1}} \log_p(\zeta_{p-1}^t-\zeta_{mp^{n+1}}^b) \overline{\varphi}(b) \\
&= - \big(1-\overline{\varphi}(p) \big) \sum_{b=1}^{f_{\varphi}} \log_p(\zeta_{p-1}^t-\zeta_{f_{\varphi}}^b) \overline{\varphi}(b),
\end{align*}
and the theorem follows.
\end{proof}
\begin{corollary} \label{fouriereval}
If $p-1 \mid t$, then
\[ \widehat{\lambda_{\chi}}(\zeta_{\psi}) = \frac{1-\overline{\varphi}(p)}{1-\varphi(p)/p} \tau(\overline{\varphi}) L_p(1,\varphi).\]
where $L_p(s,\varphi)$ is the Leopoldt-Kubota $p$-adic $L$-function. In particular, if $p \mid f_{\varphi}$, then
\[ \widehat{\lambda_{\chi}}(\zeta_{\psi}) = \tau(\overline{\varphi}) L_p(1,\varphi).\]
\end{corollary}

\begin{proof}
The first fact follows immediately from the formula (see \cite{W})
\[ L_p(1,\varphi) = - \left( 1- \frac{\varphi(p)}{p} \right) \frac{1}{\tau(\overline{\varphi})} \sum_{a=1}^{f_{\varphi}} \log_p(1-\zeta_{f_{\varphi}}^a) \overline{\varphi}(a).\]
The second fact follows from the first since if $p \mid f_{\varphi}$, then $\varphi(p)=0$.
\end{proof}

Combining the above corollary with results from Iwasawa \cite{IW} allow us to view the Gauss sums $\tau(\overline{\chi\psi})$ in an interesting light when the conductor of $\chi \psi$ is a $p$-power. Essentially, they arise as special values of the Fourier transform of a generating sequence for the projective limit of principal units of $\mbb{Q}_p(\zeta_{p^{n+1}})$.

\begin{theorem}\label{principal}
Let $k=\mbb{Q}(\zeta_p)$ and $k_{\infty}/k$ the cyclotomic $\mbb{Z}_p$-extension of $k$. Let $F = \mbb{Q}(\zeta_{p-1})$. Let $\chi$ be a non-trivial even character of $\Delta$. There exists a $p-1$-st root of unity $\zeta_{\chi} \neq 1$ such that the distribution $\upsilon_{\chi} \in \mcal{K}_{\chi}^F(\Gamma)$ associate to the norm coherent sequence $(\zeta_{\chi} - \zeta_{p^{n+1}}) \in \varprojlim (Fk_n)^{\times}$ satisfies
\[ \widehat{\upsilon_{\chi}}(\zeta_{\psi}) = \tau(\overline{\chi\psi}) \cdot \text{(unit)}\]
for every wildly ramified character $\psi$.

\end{theorem}
\begin{proof}
Let $U_n$ denote the principal units of $\mbb{Q}_p(\zeta_{p^{n+1}})$. Let $U = \varprojlim U_n$ where the projective limit is with respect to the norm maps, and let $U(\Gamma)$ denote the collection of distributions associate to the norm coherent sequences of $U$. $U$ is naturally a $\mbb{Z}_p\llbracket \Gamma \rrbracket$-module, $U(\Gamma)\subseteq \mcal{D}(\Gamma)$ is naturally an $\mcal{M}(\Gamma)$-module, and they are, in fact, isomorphic $\Lambda$-modules.

Now, let $C_n \subseteq U_n$ denote the topological closure of the cyclotomic units of $\mbb{Q}(\zeta_{p^{n+1}})$ congruent to $1$ modulo $1-\zeta_{p^{n+1}}$, and let $C = \varprojlim C_n$ with respect to the norm maps. Recall that $e_{\chi} C_n$ is generated by
\[ \left( \zeta_{p^{n+1}}^{(1-\delta_0\gamma_0)/2} \frac{\zeta_{p^{n+1}}^{\delta_0\gamma_0} -1}{\zeta_{p^{n+1}}-1} \right)^{(p-1)e_{\chi}} \]
where $\delta_0$ generates $\Delta$. The term above defines a norm coherent sequence in $e_{\chi} C$ and we let $\xi_{\chi}$ denote the associated distribution. What's more, there exists a $p-1$-st root of unity $\zeta_{\chi} \in \mbb{Z}_p$ not equal to $1$ such that $e_{\chi} U_n$ is generated by
\[ \left( \frac{\zeta_{\chi}-\zeta_{p^{n+1}}}{\zeta} \right)^{e_{\chi}} \]
where $\zeta$ is the $p-1$-st root of unity such that $\zeta_{\chi}-1 \equiv \zeta \bmod{\zeta_{p^{n+1}}-1}$ (see \cite[Theorem 13.54]{W}). Once again, the term above forms a norm coherent sequence in $e_{\chi} U$ and we let $\upsilon_{\chi}$ denote the associated distribution. From \Cref{app}, we have
\[ \widehat{\upsilon_{\chi}}(\zeta_{\psi}) = - \sum_{b=1}^{f_{\varphi}} \log_p \big( \zeta_{\chi} - \zeta_{f_{\varphi}}^b \big) \overline{\varphi(b)} \]
where $\varphi = \chi \psi$. Since the terms $\log_p(\zeta_{\chi}-\zeta_{f_{\varphi}}^b)$ for $(b,f_{\varphi})=1$ are linearly independent over $\mbb{Q}$, they must also be linearly independent over $\mbb{Q}^{\alg}$ by a theorem of Brumer. Hence $\widehat{\upsilon_{\chi}}(\zeta_{\psi}) \neq 0$.

Now, there exists a distribution $\mu_{\chi} \in \mcal{M}(\Gamma)$ such that $\xi_{\chi} = \mu_{\chi} * \upsilon_{\chi}$. Specifically, $\mu_{\chi}$ is formed from the coherent sequence of group ring elements in $\varprojlim \mbb{Z}_p\llbracket \Gamma_n \rrbracket$ that map the generator for $e_{\chi} U_n$ to the generator for $e_{\chi} C_n$. There exists $H(T) \in \Lambda^{\times}$ such that
\[ \widehat{\mu_{\chi}}(T) \cdot H(T) = G_{\chi}(T) \in \Lambda \]
where
\[ G_{\chi}\big( (1+p)^s \big) = L_p(1-s,\chi),\]
(see \cite{IW} or \cite[Theorem 13.56]{W}). Let $F_{\chi}(T) \in \Lambda$ such that $L_p(s,\chi\psi) = F_{\chi}\big( \zeta_{\psi}(1+p)^s \big)$. Then $F_{\chi}$ and $G_{\chi}$ are related via the formula
\[ F_{\chi}(T) = G_{\chi}\left( \frac{1+p}{T} \right).\]
Since $\widehat{\upsilon_{\chi}}(\zeta_{\psi}) \neq 0$, by \Cref{transform} we have that
\[ \left[ \widehat{\mu_{\chi}}(T) = \frac{\widehat{\xi_{\chi}}(T)}{\widehat{\upsilon_{\chi}}(T)} \right]_{T=\zeta}\]
for any $p^n$-th root of unity $\zeta$. It follows that
\begin{align*}
L_p(1,\chi\psi) = F_{\chi}\big( \zeta_{\psi}(1+p) \big) &= G_{\chi} \big( \zeta_{\overline{\psi}} \big)
=  \frac{\widehat{\xi_{\chi}}(\zeta_{\overline{\psi}})}{\widehat{\upsilon_{\chi}}(\zeta_{\overline{\psi}})} \cdot H(\zeta_{\overline{\psi}}).
\end{align*}
On the other hand, by \Cref{fouriereval} we have
\[ L_p(1,\chi \overline{\psi}) = \frac{\widehat{\xi_{\chi}}(\zeta_{\overline{\psi}})}{\tau(\overline{\chi}\psi)(\chi(\delta_0) \zeta_{\psi}-1)}. \]
Dividing the formula for $L_p(1,\chi\overline{\psi})$ by the formula for $L_p(1,\chi\psi)$ yields
\[ \frac{L_p(1,\chi\overline{\psi})}{L_p(1,\chi\psi)} = \frac{ \widehat{\upsilon_{\chi}}(\zeta_{\overline{\psi}})}{\tau(\overline{\chi}\psi)(\chi(\delta_0)\zeta_{\psi} -1) H(\zeta_{\overline{\psi}})} \]
whence
\[ \widehat{\upsilon_{\chi}}(\zeta_{\psi}) = \tau(\overline{\chi \psi}) \left [(\chi(\delta_0)\zeta_{\overline{\psi}} -1)H(\zeta_{\psi}) F_{\chi}\big( \zeta_{\psi}(1+p) \big)^{1-\sigma} \right] \]
where $\sigma \in \Gal(\mbb{Q}_p(\zeta_{\psi})/\mbb{Q}_p)$ is defined by $\iota: \zeta_{\psi} \mapsto \zeta_{\overline{\psi}}$. Since $\chi$ is non-trivial, the term above in brackets is a unit of $\mbb{Z}_p[\zeta_{\psi}]$.
\end{proof}

Keeping notation from the proof of \Cref{principal}, let $\mathrm{R}_p(e_{\chi} U_n)$ denote the $p$-adic regulator of $e_{\chi} U_n$. Specifically, for any set of elements $x_1,x_2,\ldots, x_{p^{n}} \in e_{\chi} U_n$ that generate $e_{\chi} U_n$ as a $\ZP$-module, set
\[ \mathrm{R}_p(e_{\chi}U_n) = \det \big( \log_p(x_j^{\gamma}) \big)_{j,\gamma} \]
where $\gamma$ ranges over $\Gamma_n$. Note that $\mathrm{R}_p(e_{\chi} U_n)$ is determined only up to a unit of $\ZP$.

\begin{corollary}\label{principalcor}
There exists $(y_n) \in U$ such that $(y_n^{e_{\chi}})$ generates $e_{\chi} U$ and the associated distribution $\upsilon_{\chi}$ satisfies
\[ \prod_{\psi \in \widehat{\Gamma_n}} \widehat{\upsilon_{\chi}}(\zeta_{\psi}) = \prod_{\psi \in \widehat{\Gamma_n}} \tau(\overline{\chi\psi}) = \mathrm{R}_p(e_{\chi} U_n).\]
\end{corollary}

\begin{proof}
Let $(\mu_n) \in \ZP\llbracket \Gamma \rrbracket$ be associate to the power series
\[ \left[ \left( \frac{\chi(\delta_0)}{T} - 1 \right) H(T) \right]^{-1} \in \Lambda^{\times} \]
so that
\[  y_n^{e_{\chi}} = \left( \frac{\zeta_{\chi} - \zeta_{p^{n+1}}}{\zeta} \right)^{\mu_n e_{\chi}} \]
also generates $e_{\chi} U_n$. Let $\upsilon_{\chi} \in \mcal{K}(\Gamma)$ be associate to the sequence $(y_n) \in U$. Note that $\widehat{\upsilon_{\chi}}(\zeta_{\psi})$ is the $\chi \psi$-part of $\mathrm{R}_p(e_{\chi} U_n)$. To be precise, let
\[ \Upsilon^{(n)} :=-\sum_{\delta \in \Delta} \sum_{a=0}^{p^n-1} \log_p(y_n^{\gamma_0^a \delta}) (\delta \gamma_0^a)^{-1} \in \mbb{Q}_p(\zeta_{p^{n+1}})[\Delta \times \Gamma_n].\]
Now let $\Upsilon_{\chi\psi}^{(n)} \in \mbb{Q}_p(\zeta_{p^{n+1}})$ be defined by $e_{\chi\psi} \Upsilon^{(n)} = \Upsilon_{\chi\psi}^{(n)} e_{\chi\psi}$. Then the regulator for $e_{\chi} U_n$ is the product
\[ \prod_{\psi \in \widehat{\Gamma_n}} \Upsilon_{\chi\psi}^{(n)} = \prod_{\psi \in \widehat{\Gamma_n} } \widehat{\upsilon_{\chi}}(\zeta_{\psi}) = \prod_{\psi \in \widehat{\Gamma_n}} \tau(\overline{\chi\psi}) \cdot F_{\chi} \big( \zeta_{\psi}(1+p) \big)^{1-\iota}= \prod_{\psi \in \widehat{\Gamma_n}} \tau(\overline{\chi\psi}).\]
\end{proof}

Under the additional assumption that $p$ is regular, the above equality of products can be refined into an equality of components. In particular, keeping notation from \Cref{principal} and \Cref{principalcor}, we have
\begin{corollary}
If $p$ is a regular prime, then there exists $(v_n) \in U$ such that $(v_n^{e_{\chi}})$ generates $e_{\chi} U$ and the associated distribution $\nu_{\chi}$ satisfies
\[ \widehat{\nu_{\chi}}(\zeta_{\psi})= \tau(\overline{\chi\psi}) = \mathrm{N}_{\chi \psi}^{(n)},\]
where $\mathrm{N}_{\chi \psi}^{(n)}$ is to $(v_n)$ as $\Upsilon_{\chi \psi}^{(n)}$ is to $(y_n)$.
\end{corollary}
\begin{proof}
If $p$ is a regular prime, then $F_{\chi}(T) \in \Lambda^{\times}$, hence
\[ G_{\chi}(T),\ G_{\chi} \left( \frac{1+p}{T} \right) \in \Lambda^{\times} \]
as well. So the values $F_{\chi}\big( \zeta_{\psi}(1+p) \big)^{1-\sigma}$ are interpolated by a power series in $\Lambda^{\times}$, namely,
\[ F_{\chi} \big( \zeta_{\psi}(1+p) \big)^{1-\sigma} = \left. \frac{G_{\chi}(T)}{G_{\chi}\big( (1+p)/T \big)} \right|_{T=\zeta_{\psi}}. \]
Let $\mu_n \in \ZP \llbracket \Gamma \rrbracket$ be associate to the power series
\[ \left[ \frac{G_{\chi}(T)}{G_{\chi}\big( (1+p)/T \big)} \right]^{-1} \]
so that
\[ v_n^{e_{\chi}} = y_n^{\mu_n e_{\chi}} \]
also generates $e_{\chi} U_n$. Let $\nu_{\chi}$ be associate to the sequence $(v_n)$. It follows that
\[ \mathrm{N}_{\chi\psi}^{(n)} = -\sum_{a=0}^{p^n-1} \overline{\psi}(a) \sum_{\delta \in \Delta} \log_p\big( v_n^{\gamma_0^a \delta} \big) \overline{\chi}(\delta) = \widehat{\nu_{\chi}}(\zeta_{\psi}) = \tau(\overline{\chi \psi}).\]
\end{proof}

\begin{remark}
Note that $\mu_{\chi} \in \mcal{M}(\Gamma)$ from the proof of \Cref{principal} can be given explicitly in terms of $\widehat{\upsilon_{\chi}}(T)$ and $\widehat{\xi_{\chi}}(T)$ from \Cref{principal}. In particular,
\[ \mu_{\chi}(a+p^n\mbb{Z}_p) = \frac{1}{p^n} \sum_{c=0}^{p^n-1} \zeta_{p^n}^{-ca} \left. \frac{\widehat{\xi_{\chi}}(T)}{\widehat{\upsilon_{\chi}}(T)} \right|_{T=\zeta_{p^n}^c}.\]
Moreover, let $E_n$ denote the units of $\mbb{Q}(\zeta_{p^{n+1}})$. The map
$E_n \to \mbb{Z}_p[\Gamma_n]$ defined by
\[ \epsilon \mapsto (\Upsilon_{\chi}^{(n)})^{-1} \cdot \sum_{a=0}^{p^n-1} \log_p\big( \epsilon^{(p-1)e_{\chi}\gamma_0^a} \big) \gamma_0^{-a} \]
is a $\Gal(\mbb{Q}(\zeta_{p^{n+1}})/\mbb{Q})$-module map where the invertibility of $\Upsilon_{\chi}^{(n)}$ follows from the non-vanishing of $\widehat{\upsilon_{\chi}}(\zeta_{\psi})$. Note that $(\Upsilon_{\chi}^{(n)})^{-1}$ is acting as an \emph{integralizer} in the sense of \cite[Definition 2.5]{A}. The image of the cyclotomic units of $\mbb{Q}(\zeta_{p^{n+1}})$ under this map annihilates the $\chi$-part of the Sylow $p$-subgroup of $\Cl \big( \mbb{Q} (\zeta_{p^{n+1}}^+) \big)$ \cite[Theorem 3.1]{A}. Let $M_{\chi}^{(n)} \in \mbb{Z}_p[\Gamma_n]$ denote the group ring element:
\[ M_{\chi}^{(n)} =\sum_{a=0}^{p^n-1} \mu_{\chi}(a+p^n\mbb{Z}_p) \gamma_0^{-a} \in \mbb{Z}_p[ \Gamma_n].\]
Likewise, let $\Upsilon_{\chi}^{(n)}$ and $\Xi_{\chi}^{(n)}$ be the group ring elements in $\mbb{Q}_p(\zeta_{p^{n+1}})[\Gamma_n]$ corresponding to the distributions $\upsilon_{\chi}$ and $\xi_{\chi}$, respectively. Since
\[ M_{\chi}^{(n)} = (\Upsilon_{\chi}^{(n)})^{-1} \cdot \Xi_{\chi}^{(n)} \]
we see that $M_{\chi}^{(n)}$ is indeed in the image of the cyclotomic units of the aforementioned map. Therefore $(M_{\chi}^{(n)}) \in \ZP \llbracket \Gamma \rrbracket$ is a coherent sequence of explicit annihilators of the $\chi$-part of the Sylow $p$-subgroup of $\Cl\big( \mbb{Q}(\zeta_{p^{n+1}}^+) \big)$.
\end{remark}

\bibliographystyle{plain}
\bibliography{mybib}

\end{document}